\documentclass[11pt]{article}
\usepackage[tbtags]{amsmath}
\usepackage{amssymb}
\usepackage{amsthm}
\usepackage[misc]{ifsym}
\usepackage{graphicx}
\usepackage{tikz}
\usetikzlibrary{shapes,arrows}
\usepackage{cases}
\numberwithin{equation}{section}
\newtheorem{definition}{Definition}[section]
\newtheorem{theorem}{Theorem}[section]
\newtheorem{lemma}{Lemma}[section]

\normalsize
\usepackage{hyperref,amsmath,amssymb,amscd}
\title{\bf A Stackelberg Game of Backward Stochastic Differential Equations with Applications
\thanks{This work is supported by National Key R\&D Program of China (Grant No. 2018YFB1305400) and National Natural Science Foundations of China (Grant No. 11571205, 11831010).}}
\author{\normalsize Yueyang Zheng\thanks{\it School of Mathematics, Shandong University, Jinan 250100, P.R. China, E-mail: zhengyueyang0106@163.com} , Jingtao Shi\thanks{\it School of Mathematics, Shandong University, Jinan 250100, P.R. China, E-mail: shijingtao@sdu.edu.cn}}
\date{}
\newtheorem{mypro}{Proposition}[section]

\newtheorem{Remark}{Remark}[section]
\begin{document}
\maketitle

\noindent{\bf Abstract:}\quad This paper is concerned with a Stackelberg game of backward stochastic differential equations (BSDEs), where the coefficients of the backward system and the cost functionals are deterministic, and the control domain is convex. Necessary and sufficient conditions of the optimality for the follower and the leader are first given for the general problem, by the stochastic maximum principles of BSDEs and forward-backward stochastic differential equations (FBSDEs), respectively. Then a linear-quadratic (LQ) Stackelberg game of BSDEs is investigated under standard assumptions. The state feedback representation for the optimal control of the follower is first given via two Riccati equations. Then the leader's problem is formulated as an optimal control problem of FBSDE with the control-independent diffusion term. Two high-dimensional Riccati equations are introduced to represent the state feedback for the optimal control of the leader. The solvability of the four Riccati equations are discussed. Theoretic results are applied to an optimal consumption rate problem of two players in the financial market.

\vspace{1mm}

\vspace{1mm}


\section{Introduction}

The Stackelberg game, also known as the leader-follower game, has been an active research topic in recent years. Among various dynamic games, the Stackelberg game has the hierarchical structure of decision making between the players, which are appealing in both theory and applications. The Stackelberg solution to the game is achieved when one of the two players is forced to wait until the other player announces his strategy, before making his own decision. Problems of such sequential nature arise frequently in economics, where decisions must be made by two parties and one of them is subordinated to the other, and hence must wait for the other party's decision before formulating its own.

The research of Stackelberg game can be traced back to the pioneering work by Stackelberg \cite{S52} in static competitive economics. Simann and Cruz \cite{SC73} studied the dynamic LQ Stackelberg differential game, and the Stackelberg solution was expressed in terms of Riccati equations. Bagchi and Basar \cite{BB81} investigated the stochastic LQ Stackelberg differential game, where the diffusion term of the state equation does not contain the state and control variables. Existence and uniqueness of its Stackelberg solution are established, and the leader's optimal strategy is solved as a nonstandard stochastic control problem and is shown to satisfy a particular integral equation. Yong \cite{Yong02} extended the stochastic LQ Stackelberg differential game to a rather general framework, where the coefficients could be random matrices, the control variables could enter the diffusion term of the state equation and the weight matrices for the controls in the cost functionals need not to be positive definite. The problem of the leader is first described as a stochastic control problem of an FBSDE. Moreover, it is shown that the open-loop solution admits a state feedback representation if a new stochastic Riccati equation is solvable. \O ksendal et al. \cite{OSU13} proved a maximum principle for the stochastic Stackelberg differential game when the noise is described as an Ito-L\'{e}vy process, and found applications to a continuous time manufacturer-newsvendor model. Bensoussan et al. \cite{BCS15} proposed several solution concepts in terms of the players' information sets, for the stochastic Stackelberg differential game with the control-independent diffusion term, and derived the maximum principle under the adapted closed-loop memoryless information structure. Xu and Zhang \cite{XZ16} studied both discrete- and continuous-time stochastic Stackelberg differential games with time delay. By introducing a new costate, a necessary and sufficient condition for the existence and uniqueness of the Stackelberg equilibrium was presented and was designed in terms of three decoupled and symmetric Riccati equations. Shi et el. \cite{SWX16} solved a stochastic leader-follower differential game with asymmetric information, where the information available to the follower is based on some sub-$\sigma$-algebra of that available to the leader. Stochastic maximum principles and verification theorems with partial information were obtained. An LQ stochastic leader-follower differential game with noisy observation was solved via measure transformation, stochastic filtering, where not all the diffusion coefficients contain the state and control variables. In Shi et al. \cite{SWX17}, an LQ stochastic Stackelberg differential game with asymmetric information was researched, where the control variables enter both diffusion coefficients of the state equation, via some forward-backward stochastic differential filtering equations (FBSDFEs). In Li and Yu \cite{LY18}, a kind of coupled FBSDE with a multilevel self-similar domination-monotonicity structure was introduced to characterize the unique equilibrium of an LQ generalized Stackelberg game with multilevel hierarchy in a closed form. Moon and Ba\c{s}ar \cite{MB18} considered a LQ mean field Stackelberg differential games with one leader and arbitrarily large number of followers. Xu et al. \cite{XSZ18} studied a leader-follower stochastic differential game  with time delay appearing in the leader¡¯s control. The open-loop solution is explicitly given in the form of the conditional expectation with respect to several symmetric Riccati equations, with the nonhomogeneous relationship between the forward variables and the backward ones obtained in the optimization problems of both the follower and the leader. Lin et al. \cite{LJZ19} considered the open-loop LQ Stackelberg game of the mean-field stochastic systems in finite horizon. A sufficient condition for the existence and uniqueness of the Stackelberg strategy in terms of the solvability of some Riccati equations is presented. Furthermore, it was shown that the open-loop Stackelberg equilibrium admits a feedback representation involving the new state and its mean. Some recent progress about Stackelberg games can be seen in a review paper by Li and Sethi \cite{LS17} and the references therein.

A BSDE is an It\^{o}'s stochastic differential equation (SDE) in which a prescribed terminal condition $y(T)=\xi$ is given. The BSDE admits an adapted solution pair $(y(\cdot),z(\cdot))$ under some conditions, where the additional term $z(\cdot)$ is required for the solutions to the equation being adapted processes. This is essentially different from forward SDEs. The linear BSDEs are initially introduced by Bismut \cite{Bis78}. The theory of general nonlinear BSDEs is established by Pardox and Peng \cite{PP90}, and Duffie and Epstein \cite{DE92}, with applications to many fields such as optimal control, partial differential equation, differential game, mathematical finance, option pricing, etc. Two recent monographs about BSDEs can be seen in Pardoux and Rascanu \cite{PR14} and Zhang \cite{Zhang17}.

Since BSDEs are well-defined dynamic systems, it is very natural and appealing to study the control and game problems involving BSDEs. The optimal control problem of BSDEs was first studied by Peng \cite{P92,P93} and El Karoui et al. \cite{EPQ97}, when solving the recursive utility maximization problems. Dokachev and Zhou \cite{DZ99} studied a nonlinear stochastic control problem of BSDEs. A necessary condition of optimality in the form of a global maximum principle as well as a sufficient condition of optimality are presented. The general result is also applied to backward LQ (BLQ) control problem and an optimal control is obtained explicitly as a feedback of the solution to a FBSDE. In Chen and Zhou \cite{CZ00}, when a general optimization model of stochastic LQ regulators with indefinite control cost weighting matrices is studied, a subproblem with a backward dynamics is proposed and studied. Lim and Zhou \cite{LZ01} studied the general BLQ control problem by the completion-of-squares technique, and the relationship between the BSDE and the forward LQ stochastic control problem. The optimal control is represented as a feedback of the entire history of the state. Huang et al. \cite{HWX09} studied a BLQ control problem with partial information, the explicit solutions for the optimal control are obtained in terms of some FBSDFEs. Zhang \cite{Zhang11} studied a BLQ control problem with random jumps. Shi \cite{Shi11} researched an optimal control problem of BSDEs with time delayed generators. Lou and Li \cite{LL13} and Li et al. \cite{LSX19} studied the BLQ control problem for mean-field case.

The differential game problem of BSDEs was initially studied by Yu and Ji \cite{YJ08}, and the explicit form of a Nash equilibrium point was given for the LQ non-zero sum case. Wang and Yu \cite{WY10} established necessary and sufficient conditions of the Nash equilibrium point for the non-zero sum stochastic differential game of BSDEs. Then Wang and Yu \cite{WY12} considered the partial information case, and a verification theorem which is a sufficient condition for the Nash equilibrium point was established. Shi and Wang \cite{SW16} studied a non-zero sum differential game of BSDEs with time-delayed generator. An Arrow's sufficient condition for the Nash equilibrium point is proved. Huang et al. \cite{HWW16} studied a backward mean-field LQ Gaussian games of weakly coupled stochastic large population system. Wang et al. \cite{WXX18} focused on a kind of LQ non-zero sum differential game of BSDEs with asymmetric information. Du et al. \cite{DHW18} studied the mean-field game of N weakly-coupled linear BSDE system. Very recently, an LQ Stackelberg game for mean-field backward stochastic systems was studied by Du and Wu \cite{DW19}.

Inspired by the above literatures, in this paper we consider a Stackelberg game of BSDEs, where the coefficients of the backward system and the cost functionals are deterministic, and the control domain is convex. The novelty of the formulation and the contribution in this paper is the following. (1) A new kind of general Stackelberg game of BSDEs is introduced and studied by the maximum principle approach, where a terminal condition $\xi$ is given in advance. (2) For the LQ case, first, two Riccati equations, a linear BSDE and a linear SDE are introduced to get the state feedback form for the optimal control of the follower. See Theorem 4.2 and the chart below it. Then, two high-dimensional Riccati equations, a linear BSDE and a linear SDE are introduced to represent the optimal control of the leader as the state feedback form. See Theorem 4.4 and the chart below it. (3) The solvability of the four Riccati equations are discussed. (4) An optimal consumption rate problem of two players in the financial market is studied, the Stackelberg equilibrium point is represented and the optimal initial wealth reserve is obtained explicitly.

The rest of this paper is organized as follows. In Section 2, the general Stackelberg game of BSDEs is formulated. And this general problem is studied in Section 3. The follower's problem of the BSDE is considered first in Subsection 3.1, while the leader's problem of the FBSDE is studied in Subsection 3.2. By the maximum principle approach, necessary and sufficient conditions for the optimal controls of the follower and the leader's are given. In Section 4, the LQ Stackelberg game of BSDEs is investigated. Via two Riccati equations, the optimal control of the follower is given in the state feedback form. And the optimal control of the leader is represented as the state feedback form by the solutions to two high-dimensional Riccati equations. The solvability of these Riccati equations is also discussed, and the optimal solution to the LQ Stackelberg game of BSDEs is derived. In Section 5, the results in the previous sections are applied to an optimal consumption rate problem of two players in the financial market. Some concluding remarks are given in Section 6.

\section{Problem Formulation}

In this paper, we use $\mathbb{R}^n$ to denote the Euclidean space of $n$-dimensional vectors, $\mathbb{R}^{n\times d}$ to denote the space of $n\times d$ matrices, and $\mathcal{S}^n$ to denote the space of $n\times n$ symmetric matrices. $\langle\cdot,\cdot\rangle$ and $|\cdot|$ are used to denote the scalar product and norm in the Euclidean space, respectively. A $\top$ appearing in the superscript of a matrix, denotes its transpose. $f_x,f_{xx}$ denote the first- and second-order partial derivatives with respect to $x$ for a differentiable function $f$, respectively.

Let $T>0$ be fixed. Consider a complete probability space $(\Omega,\mathcal{F},\mathbb{P})$ and a standard $d$-dimensional Brownian motion $W(t)$ with $W(0)=0$, which generates the filtration $\mathcal{F}_{t}=\sigma\{W(r): 0\leq r\leq t\}$ augmented by all the $\mathbb{P}$-null sets in $\mathcal{F}$. Throughout this paper, $L_{\mathcal{F}_T}^2(\Omega,\mathbb{R}^n)$ will denote the set of $\mathbb{R}^n$-valued, $\mathcal{F}_T$-measurable random vectors, $L^2_\mathcal{F}(0,T;\mathbb{R}^n)$ will denote the set of $\mathbb{R}^n$-valued, $\mathcal{F}_t$-adapted, square integrable processes on $[0,T]$, $L^2_\mathcal{F}(0,T;\mathbb{R}^{n\times d})$ will denote the set of $n\times d$-matrix-valued, $\mathcal{F}_t$-adapted, square integrable processes on $[0,T]$, and $L^\infty(0,T;\mathbb{R}^{n\times d})$ will denote the set of $n\times d$-matrix-valued, bounded functions on $[0,T]$.

Let us consider the following controlled BSDE:
\begin{equation}\label{bsde1}
\left\{
\begin{aligned}
-dy(t)&=f(t,y(t),z(t),u_1(t),u_2(t))dt-z(t)dW(t),\ t\in[0,T],\\
  y(T)&=\xi.
\end{aligned}
\right.
\end{equation}
Here $\xi\in L_{\mathcal{F}_T}^2(\Omega,\mathbb{R}^n)$ is given and $(y(\cdot),z(\cdot))\in L^2_\mathcal{F}(0,T;\mathbb{R}^n)\times L^2_\mathcal{F}(0,T;\mathbb{R}^{n\times d})$ is the state process pair. $u_1(\cdot)\in U$ is the control process of the follower, and $u_2(\cdot)\in U$ is the control process of the leader, where $U$ is a nonempty convex subset of $\mathbb{R}^k$.

We define the cost functional of the follower and the leader as
\begin{equation}\label{cf}
J_i(u_1(\cdot),u_2(\cdot);\xi)=\mathbb{E}\bigg\{\int_0^TL_i(t,y(t),z(t),u_1(t),u_2(t))dt+h_i(y(0))\bigg\},\ i=1,2,
\end{equation}
and the admissible control sets are given by
\begin{equation}
\mathcal{U}_i[0,T]=\big\{u_i(\cdot)\in L^2_\mathcal{F}(0,T;\mathbb{R}^k)\big|u_i(t)\in U,\ a.e,\ a.s.\big\},\ i=1,2,
\end{equation}
respectively.

In \eqref{bsde1} and \eqref{cf}, $f(t,y,z,u_1,u_2):[0,T]\times\mathbb{R}^n\times\mathbb{R}^{n\times d}\times\mathbb{R}^k\times\mathbb{R}^k\rightarrow\mathbb{R}^n$, $L_i(t,y,z,u_1,u_2):[0,T]\times\mathbb{R}^n\times\mathbb{R}^{n\times d}\times\mathbb{R}^k\times\mathbb{R}^k\rightarrow\mathbb{R}$, $h_i(y):\mathbb{R}^n\rightarrow\mathbb{R}$ are given functions.
Now we introduce the following assumption that will be in force throughout this paper.

{\bf(A1)} The functions $f(t,y,z,u_1,u_2)$, $L_i(t,y,z,u_1,u_2)$, $h_i(y)$ are continuous.

{\bf(A2)} The functions $f(t,y,z,u_1,u_2)$, $L_i(t,y,z,u_1,u_2)$, $h_i(y)$ are twice continuously differentiable with respect to $y,z$. Moreover, the following inequalities hold.
\begin{equation*}
\begin{split}
&|L_i(t,y,z,u_1,u_2)|+|h_i(y)|\leq C_{0}(|y|^2+|z|^2+1),\\
&|f(t,y,z,u_1,u_2)|+|L_{iy}(t,y,z,u_1,u_2)|+|L_{iz}(t,y,z,u_1,u_2)|+|h_y(y)|\leq C_1(|y|+|z|+1),\\
&|f_y(t,y,z,u_1,u_2)|+|f_{z_j}(t,y,z,u_1,u_2)|+|f_{yy}(t,y,z,u_1,u_2)|+|L_{iyy}(t,y,z,u_1,u_2)|\\
&+|L_{iz_jz_j}(t,y,z,u_1,u_2)|+|h_{iyy}(y)|\leq C_2,
\end{split}
\end{equation*}
where $C_k>0$ are constants, $k=0,1,2$, and $z_j,j=1,\cdots,d$ are the columns of the matrix $z$.

The problem studied in this paper is proposed in the following definition.

\begin{definition}
The pair $(\bar{u}_1(\cdot),\bar{u}_2(\cdot))\in\mathcal{U}_1[0,T]\times\mathcal{U}_2[0,T]$ is called an optimal solution to the Stackelberg game of BSDEs, if it satisfies the following condition:\\
(i)\ For given $\xi\in L_{\mathcal{F}_T}^2(\Omega,\mathbb{R}^n)$ and any $u_2(\cdot)\in\mathcal{U}_2[0,T]$, there exists a map $\Gamma:\mathcal{U}_2[0,T]\times L_{\mathcal{F}_T}^2(\Omega,\mathbb{R}^n)\rightarrow\mathcal{U}_1[0,T]$ such that
\begin{equation}
J_1(\Gamma(u_2(\cdot),\xi),u_2(\cdot);\xi)=\min_{u_1(\cdot)\in\mathcal{U}_1[0,T]}J_1(u_1(\cdot),u_2(\cdot);\xi).
\end{equation}
(ii)\ There exists a unique $\bar{u}_2(\cdot)\in\mathcal{U}_2[0,T]$ such that
\begin{equation}
J_2(\Gamma(\bar{u}_2(\cdot),\xi),\bar{u}_2(\cdot);\xi)=\min_{u_2(\cdot)\in\mathcal{U}_2[0,T]}J_2(\Gamma(\bar{u}_2(\cdot),\xi),u_2(\cdot);\xi).
\end{equation}
(iii)\ The optimal strategy of the follower is $\bar{u}_1(\cdot)=\Gamma(\bar{u}_2(\cdot),\xi)$.
\end{definition}

\section{The general problem}

\subsection{Optimization for the follower}

Let $\xi\in L_{\mathcal{F}_T}^2(\Omega,\mathbb{R}^n)$ be given. Giving the leader's strategy $u_2(\cdot)\in\mathcal{U}_2[0,T]$, assume that the process $\bar{u}_1(\cdot)$ is an optimal control of the follower, and $(\bar{y}(\cdot),\bar{z}(\cdot))$ be the corresponding trajectory. Let $x(\cdot)\in L^2_\mathcal{F}(0,T;\mathbb{R}^n)$ satisfy the following adjoint equation
\begin{equation}\label{sde}
\left\{
\begin{aligned}
dx(t)=&\Big\{f_y(t,\bar{y}(t),\bar{z}(t),\bar{u}_1(t),u_2(t))^\top x(t)+L_{1y}(t,\bar{y}(t),\bar{z}(t),\bar{u}_1(t),u_2(t))^\top\Big\}dt\\
      &+\sum_{i=1}^d\Big\{f_{z_i}(t,\bar{y}(t),\bar{z}(t),\bar{u}_1(t),u_2(t))^\top x(t)\\
      &\quad+L_{1z_i}(t,\bar{y}(t),\bar{z}(t),\bar{u}_1(t),u_2(t))^{\top}\Big\}dW_i(t),\ t\in[0,T],\\
 x(0)=&\ h_{1y}(\bar{y}(0))^{\top},
\end{aligned}
\right.
\end{equation}
and the Hamiltionian function $H_1:[0,T]\times\mathbb{R}^n\times\mathbb{R}^{n\times d}\times\mathbb{R}^k\times\mathbb{R}^k\times\mathbb{R}^n\rightarrow\mathbb{R}$ is defined as
\begin{equation}\label{H1}
H_1(t,y,z,u_1,u_2,x)=-\langle x,f(t,y,z,u_1,u_2)\rangle-L_1(t,y,z,u_1,u_2).
\end{equation}

The following two results are direct from Theorems 3.1 and 5.1 of \cite{DZ99}.
\begin{theorem}\label{Maximum principle for follower}
(Necessary Conditions for Optimality) Under the assumptions {\bf (A1)} and {\bf (A2)}, let $\xi\in L_{\mathcal{F}_T}^2(\Omega,\mathbb{R}^n)$. Giving the leader's strategy $u_2(\cdot)\in\mathcal{U}_2[0,T]$, let $\bar{u}_1(\cdot)$ be the optimal control of the follower and $(\bar{y}(\cdot),\bar{z}(\cdot))$ be the corresponding trajectory. Then for any $u_1(\cdot)\in U$, we have
\begin{equation}\label{Maximum condition}
\langle H_{1u_1}(t,\bar{y}(t),\bar{z}(t),\bar{u}_1(t),u_2(t),x(t)),u_1(t)-\bar{u}_1(t)\rangle\geq0,\ a.e.\ t\in[0,T],\ a.s.,
\end{equation}
where $x(\cdot)$ is the solution to the adjoint equation \eqref{sde}.
\end{theorem}
\begin{theorem}
(Sufficient Conditions for Optimality) Under the assumptions {\bf (A1)} and {\bf (A2)}, let $\xi\in L_{\mathcal{F}_T}^2(\Omega,\mathbb{R}^n)$. Giving the leader's strategy $u_2(\cdot)\in\mathcal{U}_2[0,T]$, assume the function $h_1(\cdot)$ is convex, and the function $H_1(t,\cdot,\cdot,\cdot,u_2,x)$ is concave and Lipschitz continuous. Then $\bar{u}_1(\cdot)$ is an optimal control of the follower's problem if it satisfies \eqref{Maximum condition}, where $x(\cdot)$ is the solution to \eqref{sde}.
\end{theorem}

\subsection{Optimization for the leader}

Since the follower's optimal response $\bar{u}_1(\cdot)$ can be determined by the leader, the state equation of leader has the following form as an FBSDE:
\begin{equation}\label{fbsde}
\left\{
\begin{aligned}
       dx(t)&=\Big\{f_y(t,\bar{y}(t),\bar{z}(t),\bar{u}_1(t),u_2(t))^\top x(t)+L_{1y}(t,\bar{y}(t),\bar{z}(t),\bar{u}_1(t),u_2(t))^\top\Big\}dt\\
            &\quad+\sum_{i=1}^d\Big\{f_{z_i}(t,\bar{y}(t),\bar{z}(t),\bar{u}_1(t),u_2(t))^\top x(t)+L_{1z_i}(t,\bar{y}(t),\bar{z}(t),\bar{u}_1(t),u_2(t))^{\top}\Big\}dW_i(t),\\
-d\bar{y}(t)&=f(t,\bar{y}(t),\bar{z}(t),\bar{u}_1(t),u_2(t))dt-\bar{z}(t)dW(t),\ t\in[0,T],\\
        x(0)&=h_{1y}(\bar{y}(0))^\top,\ \bar{y}(T)=\xi.
\end{aligned}
\right.
\end{equation}
Then we introduce the Hamiltonian $H_2:[0,T]\times\mathbb{R}^n\times\mathbb{R}^{n\times d}\times\mathbb{R}^k\times\mathbb{R}^k\times\mathbb{R}^n\times\mathbb{R}^n\times\mathbb{R}^{n\times d}\times\mathbb{R}^n\rightarrow\mathbb{R}$ as
\begin{equation}\label{H2}
\begin{aligned}
H_2(t,y,z,u_1,u_2,x,p,k,q)=&\ \langle p,b(t,y,z,u_1,u_2,x)\rangle+\langle k,\sigma(t,y,z,u_1,u_2,x)\\
&+\langle q,f(t,y,z,u_1,u_2)\rangle+L_2(t,y,z,u_1,u_2,x),
\end{aligned}
\end{equation}
where we denote $b(t,y,z,u_1,u_2,x)\equiv f_y(t,y,z,u_1,u_2)^\top x+L_{1y}(t,y,z,u_1,u_2)^\top$, $\sigma(t,y,z,u_1,u_2,x)\equiv\sum_{i=1}^d\big[f_{z_i}(t,y,z,u_1,u_2)^\top x+L_{1z_i}(t,y,z,u_1,u_2)^\top\big]$.

Let $\bar{u}_2(\cdot)$ be an optimal control of the leader and $(\bar{x}(\cdot),\bar{y}(\cdot),\bar{z}(\cdot))$ be the corresponding trajectory. Let $(p(\cdot),k(\cdot),q(\cdot))\in L^2_\mathcal{F}(0,T;\mathbb{R}^n)\times L^2_\mathcal{F}(0,T;\mathbb{R}^{n\times d})\times L^2_\mathcal{F}(0,T;\mathbb{R}^n)$ satisfy the following adjoint equation
\begin{equation}\label{FBSDE}
\left\{
\begin{aligned}
-dp(t)&=H_{2x}(t,\bar{y}(t),\bar{z}(t),\bar{u}_1(t),\bar{u}_2(t),\bar{x}(t),p(t),k(t),q(t))dt-k(t)dW(t),\\
 dq(t)&=H_{2y}(t,\bar{y}(t),\bar{z}(t),\bar{u}_1(t),\bar{u}_2(t),\bar{x}(t),p(t),k(t),q(t))dt\\
      &\quad+H_{2z}(t,\bar{y}(t),\bar{z}(t),\bar{u}_1(t),\bar{u}_2(t),\bar{x}(t),p(t),k(t),q(t))dW(t),\ t\in[0,T],\\
  p(T)&=0,\ q(0)=h_{2x}(\bar{y}(0)).
\end{aligned}
\right.
\end{equation}

The following result belongs to \cite{P93}.
\begin{theorem}
(Necessary Conditions for Optimality) Assuming (A1) and (A2) hold, let $\bar{u}_2(\cdot)$ be the optimal control, and $(\bar{x}(\cdot),\bar{y}(\cdot),\bar{z}(\cdot))$ be the corresponding trajectory with $\bar{y}(T)=\xi$. Then for any $u_2(\cdot)\in U$, we have
\begin{equation}\label{Maximum principle for leader}
\big\langle H_{2u_{2}}(t,\bar{y}(t),\bar{z}(t),\bar{u}_1(t),\bar{u}_2(t),\bar{x}(t),p(t),k(t),q(t)),u_2(t)-\bar{u}_2(t)\big\rangle\geq0,\ a.e.\ t\in[0,T],\ a.s.,
\end{equation}
where $(p(\cdot),k(\cdot),q(\cdot))$  is the solution to the adjoint equation \eqref{FBSDE}.
\end{theorem}

We also examine when the necessary condition of optimality \eqref{Maximum principle for leader} becomes sufficient.
\begin{theorem}
(Sufficient conditions for Optimality) Assuming (A1) and (A2) hold, let $\bar{u}_2(\cdot)$ be an admissible control, and $(\bar{x}(\cdot),\bar{y}(\cdot),\bar{z}(\cdot))$ be the corresponding trajectory with $\bar{y}(T)=\xi$. Suppose the function $h_2(\cdot)$ is convex and the function $H_2(t,\cdot,\cdot,\cdot,\bar{u}_1(t),\cdot)$ is convex, and $(q(\cdot),p(\cdot),k(\cdot))$ is the solution to \eqref{FBSDE}. Then $\bar{u}_2(\cdot)$ is an optimal control if it satisfies \eqref{Maximum principle for leader}.
\end{theorem}
\begin{proof}
It is a direct consequence of Theorem 2.2 of Shi and Wu \cite{SW2010}, without random jumps.
\end{proof}

\section{The linear quadratic problem}

\subsection{Optimization for the follower}

In this section, we consider the following controlled linear BSDE
\begin{equation}\label{BSDE}
\left\{
\begin{aligned}
-dy(t)&=\big[A(t)y(t)+B_1(t)u_1(t)+B_2(t)u_2(t)+C(t)z(t)\big]dt-z(t)dW(t),\ t\in[0,T],\\
  y(T)&=\xi,
\end{aligned}
\right.
\end{equation}
where $A(\cdot),B_1(\cdot),B_2(\cdot),C(\cdot)$ are given deterministic matrix-valued functions. We suppose the following assumption:\\
{\bf (L1)}\ $A(\cdot),C(\cdot)\in L^\infty(0,T;\mathbb{R}^{n\times n})$, $B_1(\cdot),B_2(\cdot)\in L^\infty(0,T;\mathbb{R}^{n\times k})$.

Then we define the cost functional of the follower
\begin{equation}\label{cf1}
\begin{split}
J_1(u_1(\cdot),u_2(\cdot);\xi)&=\frac{1}{2}\mathbb{E}\bigg\{\int_0^T\Big[\langle Q_1(t)y(t),y(t)\rangle+\langle R_1(t)u_1(t),u_1(t)\rangle\\
                              &\qquad\qquad+\langle S_1(t)z(t),z(t)\rangle\Big]dt+\langle G_1y(0),y(0)\rangle\bigg\}.
\end{split}
\end{equation}
And we suppose {\bf (L2)}
\begin{equation*}
\left\{
\begin{aligned}
&Q_1(\cdot),S_1(\cdot)\in L^\infty(0,T;\mathcal{S}^n),\quad Q_1(\cdot),S_1(\cdot)\geq0,\\
&R_1(\cdot)\in L^\infty(0,T;\mathcal{S}^k),\quad R_1(\cdot)>0,\quad G_1\in\mathcal{S}^n,\quad G_1\geq0.
\end{aligned}
\right.
\end{equation*}
{\bf Problem ($LQBSDE_f$)}.\ For given $\xi\in L_{\mathcal{F}_T}^2(\Omega,\mathbb{R}^n)$ and any $u_2(\cdot)\in\mathcal{U}_2[0,T]$, find a $\bar{u}_1(\cdot)\in\mathcal{U}_1[0,T]$ such that
\begin{equation}\label{pf}
J_1(\bar{u}_1(\cdot),u_2(\cdot);\xi)=\min_{u_1(\cdot)\in\mathcal{U}_1[0,T]}J_1(u_1(\cdot),u_2(\cdot);\xi).
\end{equation}

Applying Theorems 3.1 and 3.2, we first obtain the following result.

\begin{theorem}
Under the assumptions {\bf (L1)} and {\bf (L2)}, let $\xi\in L_{\mathcal{F}_T}^2(\Omega,\mathbb{R}^n)$. Giving the leader's strategy $u_2(\cdot)\in\mathcal{U}_2[0,T]$, then the optimal control of the follower $\bar{u}_1(\cdot)$ satisfies the following condition:
\begin{equation}\label{U1}
B_1(t)^\top x(t)+R_1(t)\bar{u}_1(t)=0,\ a.e.t\in[0,T],\ a.s.,
\end{equation}
where $(x(\cdot),\bar{y}(\cdot),\bar{z}(\cdot))$ is the solution to the FBSDE:
\begin{equation}\label{E1}
\left\{
\begin{aligned}
       dx(t)&=\big[A(t)^\top x(t)+Q_1(t)\bar{y}(t)\big]dt+\big[C(t)^\top x(t)+S_1(t)\bar{z}(t)\big]dW(t),\\
-d\bar{y}(t)&=\big[A(t)\bar{y}(t)+B_1(t)\bar{u}_1(t)+B_2(t)u_2(t)+C(t)\bar{z}(t)\big]dt-\bar{z}(t)dW(t),\ t\in[0,T],\\
  \bar{y}(T)&=\xi,\quad \bar{x}(0)=G_1\bar{y}(0).
\end{aligned}
\right.
\end{equation}
\end{theorem}

Moreover, we have the following relations:
\begin{equation}\label{YZ}
\left\{
\begin{aligned}
&\bar{y}(t)=-P_1(t)x(t)-\phi(t),\\
&\bar{z}(t)=-[P_1(t)S_1(t)+I]^{-1}[P_1(t)C(t)^\top x(t)+\eta(t)],
\end{aligned}
\right.
\end{equation}
where the deterministic and differentiable function $P_1(\cdot)$ is the solution to the following equation
\begin{equation}\label{P1}
\left\{
\begin{aligned}
&\dot{P}_1(t)+A(t)P_1(t)+P_1(t)A(t)^\top-P_1(t)Q_1(t)P_1(t)+B_1(t)R_1^{-1}(t)B_1(t)^\top\\
&+C(t)[P_1(t)S_1(t)+I]^{-1}P_1(t)C(t)^\top=0,\ t\in[0,T],\\
&P_1(T)=0,
\end{aligned}
\right.
\end{equation}
and the process pair $(\phi(\cdot),\eta(\cdot))$ is the solution to the following BSDE
\begin{equation}\label{bsde2}
\left\{
\begin{aligned}
-d\phi(t)&=\Big\{[A(t)-P_1(t)Q_1(t)]\phi(t)+C(t)[P_1(t)S_1(t)+I]^{-1}\eta(t)-B_2(t)u_2(t)\Big\}dt\\
         &\quad-\eta(t)dW(t),\ t\in[0,T],\\
  \phi(T)&=-\xi.
\end{aligned}
\right.
\end{equation}
Therefore, we can get the optimal control of the follower:
\begin{equation}\label{F1}
\bar{u}_1(t)=-R_1^{-1}(t)B_1^{-1}(t)^\top x(t),\ a.e.t\in[0,T],\ a.s.,
\end{equation}
with any given leader's control $u_2(\cdot)\in\mathcal{U}_2[0,T]$ through the adjoint process $x(\cdot)$ of the FBSDE \eqref{E1}.

Next, we intend to obtain the state feedback form of $\bar{u}_1(\cdot)$. Firstly, we need the following lemma.
\begin{lemma}
Under the assumptions {\bf (L1)} and {\bf (L2)}, let $(x(\cdot),\bar{y}(\cdot),\bar{z}(\cdot))$ be the solution to the Hamiltonian system \eqref{U1}-\eqref{E1}. Then
\begin{equation}\label{R1}
x(t)=P_2(t)\bar{y}(t)+\varphi(t),\ a.e.t\in[0,T],\ a.s.,
\end{equation}
where $P_2(\cdot)$ satisfies the following Riccati equation
\begin{equation}\label{P2}
\left\{
\begin{aligned}
&\dot{P}_2(t)-P_2(t)A(t)-A(t)^\top P_2(t)-Q_1(t)+P_2(t)B_1(t)R_1^{-1}(t)B_1(t)^\top P_2(t)\\
&+P_2(t)C(t)[P_1(t)S_1(t)+I]^{-1}P_1(t)C(t)^\top P_2(t)=0,\ t\in[0,T],\\
&P_{2}(0)=G_1,
\end{aligned}
\right.
\end{equation}
and $\varphi(\cdot)$ is given by the following SDE:
\begin{equation}\label{varphi}
\left\{
\begin{aligned}
 d\varphi(t)&=\Big\{\big[-P_2(t)B_1(t)R_1^{-1}(t)B_1(t)^\top-P_2(t)C(t)[P_1(t)S_1(t)+I]^{-1}P_1(t)C(t)^\top\\
            &\qquad+A(t)^\top\big]\varphi(t)+P_2(t)B_2(t)u_2(t)-P_2(t)C(t)[P_1(t)S_1(t)+I]^{-1}\eta(t)\Big\}dt\\
            &\quad+\Big\{[P_1(t)P_2(t)+I][P_1(t)S_1(t)+I]^{-1}C(t)^\top[I+P_2(t)P_1(t)]^{-1}\varphi(t)\\
            &\qquad-[P_1(t)P_2(t)+I][P_1(t)S_1(t)+I]^{-1}C(t)^\top[I+P_2(t)P_1(t)]^{-1}P_2(t)\phi(t)\\
            &\qquad+[P_2(t)-S_1(t)][P_1(t)S_1(t)+I]^{-1}\eta(t)\Big\}dW(t),\ t\in[0,T],\\
  \varphi(0)&=0.
\end{aligned}
\right.
\end{equation}
\end{lemma}
\begin{proof}
Since $(x(\cdot),\bar{y}(\cdot),\bar{z}(\cdot),\bar{u}_1(\cdot))$ is the optimal 4-tuple of the Hamiltonian system \eqref{U1}-\eqref{E1}, then substituting \eqref{F1} into the system, noting the second relation of \eqref{YZ}, we rewrite the BSDE as:
\begin{equation}\label{BSDE2}
\left\{
\begin{aligned}
-d\bar{y}(t)&=\Big\{A(t)\bar{y}(t)-B_1(t)R_1^{-1}(t)B_1(t)^\top x(t)+B_2(t)u_2(t)\\
            &\qquad-C(t)[P_1(t)S_1(t)+I]^{-1}[P_1(t)C(t)^\top x(t)+\eta(t)]\Big\}dt\\
            &\quad+[P_1(t)S_1(t)+I]^{-1}[P_1(t)C(t)^\top x(t)+\eta(t)]dW(t),\ t\in[0,T],\\
  \bar{y}(T)&=\xi.
\end{aligned}
\right.
\end{equation}
Applying It\^o's formula to \eqref{R1}, where we let $\varphi(\cdot)$ satisfy
\begin{equation}\label{A2}
\left\{
\begin{aligned}
d\varphi(t)&=\alpha(t)dt+\beta(t)dW(t),\ t\in[0,T],\\
 \varphi(0)&=0,
\end{aligned}
\right.
\end{equation}
and make a comparison with system \eqref{E1}, we get
\begin{equation}\label{A1}
\left\{
\begin{aligned}
&\Big\{\dot{P}_2(t)-P_2(t)A(t)-A(t)^\top P_2(t)-Q_1(t)+P_2(t)B_1(t)R_1^{-1}(t)B_1(t)^\top P_2(t)\\
&\quad+P_2(t)C(t)[P_1(t)S_1(t)+I]^{-1}P_1(t)C(t)^\top P_2(t)\Big\}\bar{y}(t)+\Big\{P_2(t)B_1(t)R_1^{-1}(t)B_1(t)^\top\\
&\quad+P_2(t)C(t)[P_1(t)S_1(t)+I]^{-1}P_1(t)C(t)^\top-A(t)^\top\Big\}\varphi(t)-P_2(t)B_2(t)u_2(t)\\
&\ +P_2(t)C(t)[P_1(t)S_1(t)+I]^{-1}\eta(t)+\alpha(t)=0,\\
&\beta(t)=[P_2(t)-S_1(t)][P_1(t)S_1(t)+I]^{-1}[P_1(t)C(t)^\top x(t)+\eta(t)]\\
&\qquad\quad+C(t)^\top P_2(t)\bar{y}(t)+C(t)^\top\varphi(t).
\end{aligned}
\right.
\end{equation}
Therefore, it implies that $P_2(\cdot)$ satisfies the Riccati equation \eqref{P2}. From the first relation in \eqref{YZ} and \eqref{R1}, we have
\begin{equation}\label{A3}
x(t)=-P_2(t)[I+P_2(t)P_1(t)]^{-1}\phi(t)+[I+P_2(t)P_1(t)]^{-1}\varphi(t).
\end{equation}
Thus, combining \eqref{YZ}, \eqref{A2}, \eqref{A1} and \eqref{A3}, we get that $\varphi(\cdot)$ is given by \eqref{varphi}. The proof is complete.
\end{proof}
\begin{Remark}
{\rm We introduce, in Theorem 4.1 and Lemma 4.1, two Riccati equations for $P_1(\cdot)$ and $P_2(\cdot)$ to build the relation between $\bar{y}(\cdot)$ and $x(\cdot)$. Similarly to \cite{LZ01}, we can obtain the unique solvability of these two Riccati equations, and the unique solvability of \eqref{bsde2} and \eqref{varphi}, with the solutions $(\phi(\cdot),\eta(\cdot))$ and $\varphi(\cdot)$ respectively, is evident as they are linear BSDE and SDE with bounded deterministic coefficients and square integrable nonhomogeneous terms.}\
\end{Remark}
Based on the above lemma, we obtain the optimal state feedback of $\bar{u}_1(\cdot)$ in the follower's problem.
\begin{theorem}
Under the assumptions {\bf (L1)} and {\bf (L2)}, for any given $\xi\in L_{\mathcal{F}_T}^2(\Omega,\mathbb{R}^n)$ and $u_2(\cdot)\in\mathcal{U}_2[0,T]$, the problem ($LQBSDE_f$) is solvable with the optimal strategy $\bar{u}_1(\cdot)$ being of a feedback representation
\begin{equation}\label{FC1}
\bar{u}_1(t)=-R_1^{-1}(t)B_1(t)^\top[P_2(t)\bar{y}(t)+\varphi(t)],
\end{equation}
where $P_2(\cdot)$ and $\varphi(\cdot)$ are the solutions to \eqref{P2} and \eqref{varphi}, respectively. The optimal state trajectory $(\bar{y}(\cdot),\bar{z}(\cdot))$ is the unique solution to the BSDE:
\begin{equation}\label{bsde3}
\left\{
\begin{aligned}
-d\bar{y}(t)&=\Big\{\big[A(t)-B_1(t)R_1^{-1}(t)B_1(t)^\top P_2(t)\big]\bar{y}(t)-B_1(t)R_1^{-1}(t)B_1(t)^\top\varphi(t)\\
            &\qquad+B_2(t)u_2(t)+C(t)\bar{z}(t)\Big\}dt-\bar{z}(t)dW(t),\ t\in[0,T],\\
  \bar{y}(T)&=\xi.
\end{aligned}
\right.
\end{equation}
\end{theorem}
\begin{proof}
For given $\xi$ and $u_2(\cdot)$, let $P_1(\cdot)$ satisfy \eqref{P1}, by the standard BSDE theory we can solve \eqref{bsde2} to obtain $(\phi(\cdot),\eta(\cdot))$. Let $P_2(\cdot)$ satisfy \eqref{P2}, and by the standard SDE theory we can solve \eqref{varphi} to obtain $\varphi(\cdot)$. Thus the feedback representation \eqref{FC1} can be obtained from \eqref{A3} and \eqref{YZ}. The proof is complete.
\end{proof}
We can draw a flow chart to show the process in which we could represent the optimal feedback $\bar{u}_1(\cdot)$.\\

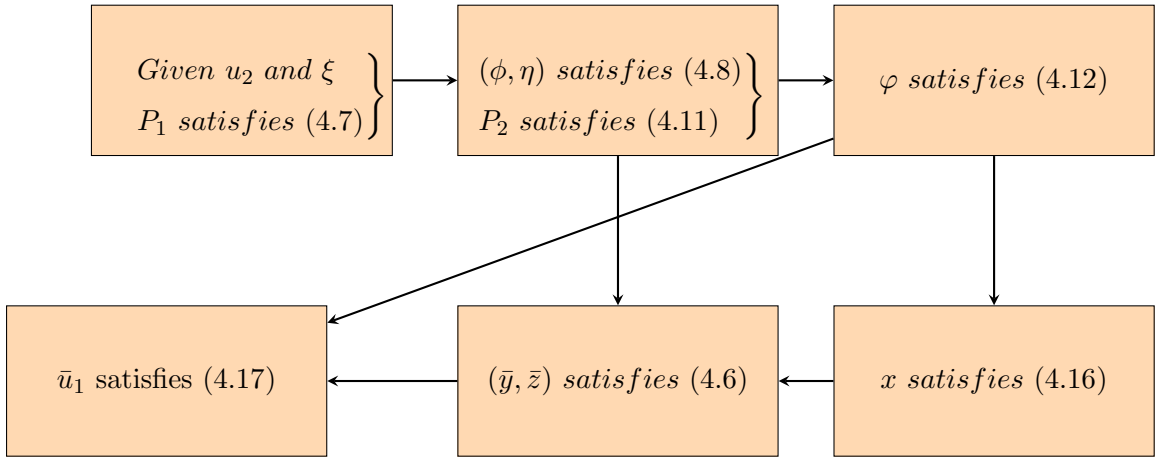
\tikzstyle{startstop} = [rectangle,minimum width=4cm,minimum height=2cm,text centered,text width =3cm,draw=black,fill=orange!30]
\tikzstyle{process} = [rectangle,minimum width=3cm,minimum height=2cm,text centered,text width =4cm,draw=black,fill=orange!30]
\tikzstyle{arrow} = [thick,->,>=stealth]
\tikzstyle{arrow2} = [thick,<-,>=stealth]
\begin{tikzpicture}[node distance=1cm]
\node (start) [startstop]
{
$$\left.
\begin{aligned}
&Given\ u_2\ and\ \xi\\
&P_1\ satisfies\ \eqref{P1}
\end{aligned}
\right\}$$
};
\node (process1) [process,right of =start,xshift=4cm]
{
$$\left.
\begin{aligned}
&(\phi,\eta)\ satisfies\ \eqref{bsde2}\\
&P_2\ satisfies\ \eqref{P2}
\end{aligned}
\right\}$$
};
\node (process2) [process,right of =process1,xshift=4cm] {$\varphi\ satisfies\ \eqref{varphi}$};
\node (process3) [process,below of =process2,yshift=-3cm] {$x\ satisfies\ \eqref{A3}$};
\node (process4) [process,left of=process3, below of =process1 ,xshift=1cm,yshift=-3cm] {$(\bar{y},\bar{z})\ satisfies\ \eqref{YZ}$};
\node (process5) [process, left of =process4, below of=process2, xshift=-10cm,yshift=-3cm] {$\bar{u}_1$\ satisfies\ \eqref{FC1}};
\draw [arrow] (start) --(process1);
\draw [arrow] (process1) --(process2);
\draw [arrow] (process2) --(process3);
\draw [arrow] (process3) --(process4);
\draw [arrow] (process1) --(process4);
\draw [arrow] (process2) -- (process5);
\draw [arrow] (process4) -- (process5);
\end{tikzpicture}

\subsection{Optimization for the leader}

In the above subsection, for any given $\xi$ and $u_2(\cdot)$, we have obtained the feedback form \eqref{FC1} of the follower's optimal control $\bar{u}_1(\cdot)$. Now we let problem ($LQBSDE_f$) be uniquely solvable for any given $(\xi,u_2(\cdot))\in L_{\mathcal{F}_T}^2(\Omega;\mathbb{R}^n)\times\mathcal{U}_2[0,T]$. Since the follower's optimal strategy $\bar{u}_1(\cdot)$ of feedback form \eqref{FC1} can be determined by the leader, the leader could take it into account in finding and announcing his optimal strategy. Consequently, from \eqref{varphi} and \eqref{bsde3}, noting \eqref{YZ},  the leader has the following state equation:
\begin{equation}\label{LeaderFBSDE}
\left\{
\begin{aligned}
 d\varphi(t)&=\Big\{\big[A(t)^\top-P_2(t)B_1(t)R_1^{-1}(t)B_1(t)^\top\big]\varphi(t)+P_2(t)C(t)[P_1(t)S_1(t)+I]^{-1}\\
            &\qquad\times P_1(t)C(t)^\top P_2(t)\bar{y}(t)+P_2(t)C(t)\bar{z}(t)+P_2(t)B_2(t)u_2(t)\Big\}dt\\
            &\quad+\Big\{\Big[[P_1(t)P_2(t)+I][P_1(t)S_1(t)+I]^{-1}C(t)^\top-[P_2(t)-S_1(t)][P_1(t)S_1(t)+I]^{-1}\\
            &\qquad\times P_1(t)C(t)^\top\Big]\varphi(t)+\Big[[P_1(t)P_2(t)+I][P_1(t)S_1(t)+I]^{-1}C(t)^\top P_2(t)\\
            &\qquad-[P_2(t)-S_1(t)][P_1(t)S_1(t)+I]^{-1}P_1(t)C(t)^\top P_2(t)\Big]\bar{y}(t)\\
            &\qquad-[P_2(t)-S_1(t)]\bar{z}(t)\Big\}dW(t),\\
-d\bar{y}(t)&=\Big\{\big[A(t)-B_1(t)R_1^{-1}(t)B_1(t)^\top P_2(t)\big]\bar{y}(t)-B_1(t)R_1^{-1}(t)B_1(t)^\top\varphi(t)\\
            &\qquad+B_2(t)u_2(t)+C(t)\bar{z}(t)\Big\}dt-\bar{z}(t)dW(t),\ t\in[0,T],\\
  \bar{y}(T)&=\xi,\ \varphi(0)=0.
\end{aligned}
\right.
\end{equation}
For any given $\xi$ and $u_2(\cdot)$, from the proof of Theorem 4.2, the solvability for the solution $(\varphi(\cdot),\bar{y}(\cdot),\bar{z}(\cdot))$ to \eqref{LeaderFBSDE} can be guaranteed though it is fully coupled.

The leader would like to choose his control $u_2(\cdot)\in\mathcal{U}_2[0,T]$ such that his cost functional
\begin{equation}\label{cf2}
\begin{aligned}
J_2(\bar{u}_1(\cdot),u_2(\cdot);\xi)&=\frac{1}{2}\mathbb{E}\bigg\{\int_0^T\Big[\langle Q_2(t)\bar{y}(t),\bar{y}(t)\rangle+\langle R_2(t)u_2(t),u_2(t)\rangle\\
                                    &\qquad\qquad+\langle S_2(t)\bar{z}(t),\bar{z}(t)\rangle\Big]dt+\langle G_{2}\bar{y}(0),\bar{y}(0)\rangle\bigg\}
\end{aligned}
\end{equation}
is minimized. And we suppose {\bf (L3)}
\begin{equation*}
\left\{
\begin{aligned}
&Q_2(\cdot),S_2(\cdot)\in L^\infty(0,T;\mathcal{S}^n),\quad Q_2(\cdot),S_2(\cdot)\geq0,\\
&R_2(\cdot)\in L^\infty(0,T;\mathcal{S}^k),\quad R_2(\cdot)>0,\quad G_2\in\mathcal{S}^n,\quad G_2\geq0.
\end{aligned}
\right.
\end{equation*}

The optimal control problem for the leader can be stated as follows.\\
{\bf Problem ($LQFBSDE_l$)}. For given $\xi\in L_{\mathcal{F}_T}^2(\Omega,\mathbb{R}^n)$, find a $\bar{u}_2(\cdot)\in\mathcal{U}_2[0,T]$ such that
\begin{equation}
J_2(\bar{u}_1(\cdot),\bar{u}_2(\cdot);\xi)=\min_{u_2(\cdot)\in\mathcal{U}_2[0,T]}J_2(\bar{u}_1(\cdot),u_2(\cdot);\xi).
\end{equation}

Applying Theorems 3.3 and 3.4, we first obtain the following result.

\begin{theorem}
Let the assumptions {\bf (L1)}, {\bf (L2)} and {\bf (L3)} hold. Let $(\bar{\varphi}(\cdot),\bar{y}(\cdot),\bar{z}(\cdot),\bar{u}_2(\cdot))$ be the optimal 4-tuple of system \eqref{LeaderFBSDE} for the terminal state $\xi\in L_{\mathcal{F}_T}^2(\Omega,\mathbb{R}^n)$. Then we have the following stationary condition:
\begin{equation}\label{OCL}
B_2(t)^\top P_2(t)p(t)+B_2(t)^\top q(t)+R_2(t)\bar{u}_2(t)=0,\ a.e.t\in[0,T],\ a.s.,
\end{equation}
where $(p(\cdot),k(\cdot),q(\cdot))$ satisfies the following adjoint FBSDE
\begin{equation}\label{Ad1}
\left\{
\begin{aligned}
 dq(t)&=\Big\{P_2(t)C(t)P_1(t)[P_1(t)S_1(t)+I]^{-1}C(t)^\top P_2(t)p(t)+\big[A(t)^\top-P_2(t)B_1(t)\\
      &\qquad\times R_1^{-1}(t)B_1(t)^\top\big]q(t)+\big[P_2(t)C(t)[P_1(t)S_1(t)+I]^{-1}[P_1(t)P_2(t)+I]\\
      &\qquad-P_2(t)C(t)P_1(t)[P_1(t)S_1(t)+I]^{-1}[P_2(t)-S_1(t)]\big]k(t)+Q_2(t)\bar{y}(t)\Big\}dt\\
      &\quad+\Big\{C(t)^\top P_2(t)p(t)+C(t)^\top q(t)-[P_2(t)-S_1(t)]k(t)+S_2(t)\bar{z}(t)\Big\}dW(t),\\
-dp(t)&=\Big\{\big[A(t)-B_1(t)R_1^{-1}(t)B_1(t)^\top P_2(t)\big]p(t)-B_1(t)R_1^{-1}(t)B_1(t)^\top q(t)\\
      &\qquad+\big[C(t)[P_1(t)S_1(t)+I]^{-1}[P_1(t)P_2(t)+I]-C(t)P_1(t)[P_1(t)S_1(t)+I]^{-1}\\
      &\qquad\times[P_2(t)-S_1(t)]\big]k(t)\Big\}dt-k(t)dW(t),\ t\in[0,T],\\
  q(0)&=G_2\bar{y}(0),\ p(T)=0.
\end{aligned}
\right.
\end{equation}
\end{theorem}

\begin{Remark}
\rm For notations simplicity, we still denote the optimal state $(\bar{y}(\cdot),\bar{z}(\cdot))$ of the leader as that in the follower's problem, while denote the third (auxiliary) optimal state of the leader as $\bar{\varphi}(\cdot)$ when in the follower's problem $\varphi(\cdot)$ is the adjoint process.
\end{Remark}

Similar to \cite{Yong02}, here we use the same technique to solve the optimization problem for the leader. We expect to obtain some kind of state feedback representation for the optimal control $\bar{u}_2(\cdot)$ via a certain Riccati equation. To make the problem clearer, let us put \eqref{LeaderFBSDE}, \eqref{OCL} and \eqref{Ad1} together:
\begin{equation}\label{FBSDEs}
\left\{
\begin{aligned}
d\bar{\varphi}(t)&=\Big\{\big[A(t)^\top-P_2(t)B_1(t)R_1^{-1}(t)B_1(t)^\top\big]\bar{\varphi}(t)+P_2(t)C(t)[P_1(t)S_1(t)+I]^{-1}\\
                 &\qquad\times P_1(t)C(t)^\top P_2(t)\bar{y}(t)+P_2(t)C(t)\bar{z}(t)+P_2(t)B_2(t)\bar{u}_2(t)\Big\}dt\\
                 &\quad+\Big\{\Big[[P_1(t)P_2(t)+I][P_1(t)S_1(t)+I]^{-1}C(t)^\top-[P_2(t)-S_1(t)]\\
                 &\qquad\times [P_1(t)S_1(t)+I]^{-1}P_1(t)C(t)^\top\Big]\bar{\varphi}(t)+\Big[[P_1(t)P_2(t)+I]\\
                 &\qquad\times[P_1(t)S_1(t)+I]^{-1}C(t)^\top P_2(t)-[P_2(t)-S_1(t)][P_1(t)S_1(t)+I]^{-1}\\
                 &\qquad\times P_1(t)C(t)^\top P_2(t)\Big]\bar{y}(t)-[P_2(t)-S_1(t)]\bar{z}(t)\Big\}dW(t),\\
            dq(t)&=\Big\{P_2(t)C(t)P_1(t)[P_1(t)S_1(t)+I]^{-1}C(t)^\top P_2(t)p(t)+\big[A(t)^\top-P_2(t)B_1(t)\\
                 &\qquad\times R_1^{-1}(t)B_1(t)^\top\big]q(t)+\big[P_2(t)C(t)[P_1(t)S_1(t)+I]^{-1}[P_1(t)P_2(t)+I]\\
                 &\qquad-P_2(t)C(t)P_1(t)[P_1(t)S_1(t)+I]^{-1}[P_2(t)-S_1(t)]\big]k(t)+Q_2(t)\bar{y}(t)\Big\}dt\\
                 &\quad+\Big\{C(t)^\top P_2(t)p(t)+C(t)^\top q(t)-[P_2(t)-S_1(t)]k(t)+S_2(t)\bar{z}(t)\Big\}dW(t),\\
     -d\bar{y}(t)&=\Big\{\big[A(t)-B_1(t)R_1^{-1}(t)B_1(t)^\top P_2(t)\big]\bar{y}(t)-B_1(t)R_1^{-1}(t)B_1(t)^\top\bar{\varphi}(t)\\
                 &\qquad+B_2(t)\bar{u}_2(t)+C(t)\bar{z}(t)\Big\}dt-\bar{z}(t)dW(t),\\
           -dp(t)&=\Big\{\big[A(t)-B_1(t)R_1^{-1}(t)B_1(t)^\top P_2(t)\big]p(t)-B_1(t)R_1^{-1}(t)B_1(t)^\top q(t)\\
                 &\qquad+\big[C(t)[P_1(t)S_1(t)+I]^{-1}[P_1(t)P_2(t)+I]-C(t)P_1(t)[P_1(t)S_1(t)+I]^{-1}\\
                 &\qquad\times[P_2(t)-S_1(t)]\big]k(t)\Big\}dt-k(t)dW(t),\ t\in[0,T],\\
 \bar{\varphi}(0)&=0,\ q(0)=G_2\bar{y}(0),\ \bar{y}(T)=\xi,\ p(T)=0,\\
\end{aligned}
\right.
\end{equation}
Note that the equations for $(\bar{y}(\cdot),\bar{z}(\cdot),\bar{\varphi}(\cdot))$ form a coupled FBSDE, and those for $(p(\cdot),k(\cdot),q(\cdot))$ form another coupled FBSDE. These two FBSDEs are further coupled through their initial and/or terminal conditions and \eqref{OCL}. Hence, the above is a coupled system of FBSDEs. We may look at the above in a different way. To this end, let us set (The time variable $t$ is omitted.)
\begin{equation}\label{definition}
X=
\begin{pmatrix}
\bar{\varphi}\\
q
\end{pmatrix},\ \
Y=
\begin{pmatrix}
p\\
\bar{y}
\end{pmatrix},\ \
Z=
\begin{pmatrix}
k\\
\bar{z}
\end{pmatrix},\ \
\hat{G}_2=
\begin {pmatrix}
0&0\\
0&G_2
\end{pmatrix},\ \
\hat{\xi}=
\begin {pmatrix}
0\\
\xi
\end{pmatrix},
\end{equation}
and
\begin{equation*}
\hat{A}_1=
\begin{pmatrix}
A-B_1R_1^{-1}B_1^\top P_2&0\\
0&A-B_1R_1^{-1}B_1^\top P_2
\end{pmatrix},\ \
\hat{B}_1=
\begin{pmatrix}
P_2B_2\\
0
\end{pmatrix},\ \
\hat{B}_2=
\begin{pmatrix}
0\\
B_2
\end{pmatrix},\ \
\end{equation*}
\begin{equation*}
\hat{C}_1=
\begin{pmatrix}
(P_1P_2+I)(P_1S_1+I)^{-1}C^\top-(P_2-S_1)(P_1P_2+I)^{-1}P_1C^\top&0\\
0&C^\top
\end{pmatrix}
\end{equation*}
\begin{equation*}
\hat{D}_1=
\begin{pmatrix}
0&P_2C\\
P_2C(P_1S_1+I)^{-1}(P_1P_2+I)-P_2CP_1(P_1S_1+I)^{-1}(P_2-S_1)&0
\end{pmatrix},
\end{equation*}
\begin{equation*}
\end{equation*}\begin{equation*}
\hat{F}_1=
\begin {pmatrix}
0&P_2C(P_1S_1+I)^{-1}P_1C^\top P_2\\
P_2CP_1(P_1S_1+I)^{-1}C^\top P_2&Q_2
\end{pmatrix},\ \
\end{equation*}
\begin{equation*}
\hat{F}_{2}=
\begin {pmatrix}
0&-B_1R_1^{-1}B_1^\top\\
-B_1R_1^{-1}B_1^\top&0
\end{pmatrix},\ \
\hat{S}_1=
\begin {pmatrix}
0&-(P_2-S_1)\\
-(P_2-S_1)&S_2
\end{pmatrix},
\end{equation*}
then \eqref{FBSDEs} is equivalent to the FBSDE:
\begin{equation}\label{KWFBSDE}
\left\{
\begin{aligned}
 dX(t)&=\big[\hat{A}_1^\top X(t)+\hat{F}_1Y(t)+\hat{D}_1Z(t)+\hat{B}_1\bar{u}_2(t)\big]dt\\
      &\quad+\big[\hat{C}_1X(t)+\hat{D}_1^\top Y(t)+\hat{S}_1Z(t)\big]dW(t),\\
-dY(t)&=\big[\hat{F}_2X(t)+\hat{A}_1Y(t)+\hat{C}_1^\top Z(t)+\hat{B}_2\bar{u}_2(t)\big]dt-Z(t)dW(t),\ t\in[0,T],\\
  X(0)&=\hat{G}_2Y(0),\ Y(T)=\hat{\xi},\\
&\hspace{-1cm}\hat{B}_1^\top Y(t)+\hat{B}_2^\top X(t)+R_2(t)\bar{u}_2(t)=0,\ a.e.t\in[0,T],\ a.s.
\end{aligned}
\right.
\end{equation}

Noting that \eqref{KWFBSDE} is a coupled FBSDE with initial condition coupling, while the diffusion term of the forward equation is control independent. This is a new feature different from that in the (forward) leader-follower differential game studied in \cite{Yong02}. In the following, we try to decouple the above FBSDE using the similar technique as in Subsection 4.1.

Suppose we have the relation
\begin{equation}\label{YX}
\left\{
\begin{aligned}
&Y(t)=-\Pi_1(t)X(t)-\tilde{\phi}(t),\quad \Pi_1(T)=0,\\
&-d\tilde{\phi}(t)=\tilde{\alpha}(t)dt-\tilde{\eta}(t)dW(t),\quad\tilde{\phi}(T)=-\hat{\xi}.
\end{aligned}
\right.
\end{equation}
Applying It\^{o}'s formula to $Y(\cdot)$, and make comparison with system \eqref{KWFBSDE}, we have
\begin{equation}\label{Z1}
\left\{
\begin{aligned}
&\hat{F}_2X(t)+\hat{A}_1Y(t)+\hat{C}_1^\top Z(t)+\hat{B}_2\bar{u}_2(t)-\Pi_1\hat{A}_1^\top X(t)-\Pi_1\hat{F}_1Y(t)\\
&-\Pi_1\hat{D}_1Z(t)-\Pi_1\hat{B}_1\bar{u}_2(t)-\dot{\Pi}X(t)+\tilde{\alpha}(t)=0,\\
&Z(t)+\Pi_1\hat{C}_1X(t)+\Pi_1\hat{D}_1^\top Y(t)+\Pi_1\hat{S}_1Z(t)+\tilde{\eta}(t)=0.
\end{aligned}
\right.
\end{equation}
Noting the last relation in \eqref{KWFBSDE}, it implies the following equation of $\Pi_1(\cdot)$ and $\tilde{\phi}(\cdot)$:
\begin{equation}\label{Pi1}
\left\{
\begin{aligned}
&\dot{\Pi}_1+\hat{A}_1\Pi_1+\Pi_1\hat{A}_1^\top-\Pi_1\hat{F}_1\Pi_1+(\Pi_1\hat{B}_1-\hat{B}_2)R_2^{-1}(\hat{B}_1^\top\Pi_1-\hat{B}_2^\top)\\
&+(\hat{C}_1^\top-\Pi_1\hat{D}_1)(I+\Pi_1\hat{S}_1)^{-1}\Pi_1(\hat{C}_1-\hat{D}_1^\top\Pi_1)-\hat{F}_{2}=0,\\
&\Pi_1(T)=0,
\end{aligned}
\right.
\end{equation}
\begin{equation}\label{phi}
\left\{
\begin{aligned}
-d\tilde{\phi}(t)&=\Big\{\big[\hat{A}_1-\Pi_1\hat{F}_1+(\Pi_1\hat{B}_1-\hat{B}_2)R_2^{-1}\hat{B}_1^\top+(\Pi_1\hat{D}_1-\hat{C}_1^\top)(I+\Pi_1\hat{S}_1)^{-1}\Pi_1\hat{D}_1^\top\big]\tilde{\phi}(t)\\
                 &\qquad-(\Pi_1\hat{D}_1-\hat{C}_1^\top)(I+\Pi_1\hat{S}_1)^{-1}\tilde{\eta}(t)\Big\}dt-\tilde{\eta}(t)dW(t),\\
  \tilde{\phi}(T)&=-\hat{\xi}.
\end{aligned}
\right.
\end{equation}

Then we suppose
\begin{equation}\label{XY}
\left\{
\begin{aligned}
&X(t)=\Pi_{2}(t)Y(t)+\tilde{\varphi}(t),\ \Pi_2(0)=\hat{G}_2,\\
&d\tilde{\varphi}(t)=\tilde{\beta}(t)dt+\tilde{\gamma}(t)dW(t),\ \tilde{\varphi}(0)=0.
\end{aligned}
\right.
\end{equation}
Applying the It\^{o}'s formula to $X(\cdot)$, and make comparison with system \eqref{KWFBSDE} we have
\begin{equation}\label{Z2}
\left\{
\begin{aligned}
&\hat{A}_1^\top X(t)+\hat{F}_1Y(t)+\hat{D}_1Z(t)+\hat{B}_1\bar{u}_2(t)+\Pi_2\hat{F}_2X(t)+\Pi_2\hat{A}_1Y(t)+\Pi_2\hat{C}_1^\top Z(t)\\
&+\Pi_2\hat{B}_2\bar{u}_2(t)-\dot{\Pi}_2Y(t)-\tilde{\beta}(t)=0,\\
&\Pi_2Z(t)+\tilde{\gamma}(t)-\hat{C}_1X(t)-\hat{D}_1^\top Y(t)-\hat{S}_1Z(t)=0.
\end{aligned}
\right.
\end{equation}
From the above relationship between $X(\cdot)$ and $Y(\cdot)$ in \eqref{YX} and \eqref{XY}, we can obtain
\begin{equation}\label{X1Y1}
\left\{
\begin{aligned}
X(t)&=(I+\Pi_2\Pi_1)^{-1}\big[-\Pi_2\tilde{\phi}(t)+\tilde{\varphi}(t)\big],\\
Y(t)&=-(I+\Pi_1\Pi_2)^{-1}\big[\Pi_1\tilde{\varphi}(t)+\tilde{\phi}(t)\big].
\end{aligned}
\right.
\end{equation}
Combining this with the relations with regard to $Z(\cdot)$ in \eqref{Z1}, \eqref{Z2}, we can get
\begin{equation}\label{Pi2}
\left\{
\begin{aligned}
\dot{\Pi}_2(t)&=\Pi_{2}\hat{A}_1+\hat{A}_1^\top\Pi_2+\Pi_2\hat{F}_2\Pi_2-(\hat{B}_1+\Pi_2\hat{B}_2)R_2^{-1}(\hat{B}_1+\Pi_2\hat{B}_2)^\top\\
              &\quad-(\hat{D}_1+\Pi_2\hat{C}_1^\top)(I+\Pi_1\hat{S}_1)^{-1}\Pi_1(\hat{D}_1^\top+\hat{C}_1\Pi_2)+\hat{F}_{1},\\
    \Pi_{2}(0)&=\hat{G}_2,
\end{aligned}
\right.
\end{equation}
and $\tilde{\varphi}(\cdot)$ satisfy the following equation
\begin{equation}\label{tildevarphi}
\left\{
\begin{aligned}
d\tilde{\varphi}(t)&=\tilde{\beta}(t)dt+\tilde{\gamma}(t)dW(t),\\
\tilde{\varphi}(0)&=0,
\end{aligned}
\right.
\end{equation}
with
\begin{equation}\label{betagamma}
\left\{
\begin{aligned}
 \tilde{\beta}(t)&=\Big\{\hat{A}_1^\top+\Pi_2\hat{F}_2-(\hat{B}_1+\Pi_2\hat{B}_2)R_2^{-1}\hat{B}_2^\top-(\hat{D}_1+\Pi_2\hat{C}_1^\top)(I+\Pi_1\hat{S}_1)^{-1}\Pi_1\hat{C}_1\Big\}\tilde{\varphi}(t)\\
                 &\quad-(\hat{D}_1+\Pi_2\hat{C}_1^\top)(I+\Pi_1\hat{S}_1)^{-1}\tilde{\eta}(t),\\
\tilde{\gamma}(t)&=-\Big\{\hat{D}_1^\top(I+\Pi_1\Pi_2)^{-1}\Pi_1+(\Pi_2-\hat{S}_1)(I+\Pi_1\hat{S}_1)^{-1}\Pi_1\hat{D}_1^\top(I+\Pi_2\Pi_1)^{-1}\Pi_1\\
                 &\qquad-\hat{C}_1(I+\Pi_2\Pi_1)^{-1}-(\Pi_2-\hat{S}_1)(I+\Pi_1\hat{S}_1)^{-1}\Pi_1\hat{C}_1(I+\Pi_2\Pi_1)^{-1}\Big\}\tilde{\varphi}(t)\\
                 &\quad-\Big\{\hat{D}_1^\top(I+\Pi_1\Pi_2)^{-1}+(\Pi_2-\hat{S}_1)(I+\Pi_1\hat{S}_1)^{-1}\Pi_1\hat{D}_1^\top(I+\Pi_2\Pi_1)^{-1}\\
                 &\qquad+\hat{C}_1(I+\Pi_2\Pi_1)^{-1}\Pi_2+(\Pi_2-\hat{S}_1)(I+\Pi_1\hat{S}_1)^{-1}\Pi_1\hat{C}_1(I+\Pi_2\Pi_1)^{-1}\Pi_2\Big\}\tilde{\phi}(t)\\
                 &\quad+(\Pi_2-\hat{S}_1)(I+\Pi_1\hat{S}_1)^{-1}\tilde{\eta}(t).
\end{aligned}
\right.
\end{equation}

We have the following result.
\begin{theorem}
Under assumptions {\bf (L1)}, {\bf (L2)} and {\bf (L3)}, suppose the Riccati equations \eqref{Pi1} and \eqref{Pi2} admit differentiable solutions $\Pi_1(\cdot)$ and $\Pi_2(\cdot)$, respectively. Then {\bf Problem ($LQFBSDE_l$)} is solvable with the optimal strategy $\bar{u}_2(\cdot)$ being of a feedback representation
\begin{equation}\label{LC2}
\bar{u}_2(t)=-R_2^{-1}(t)\big[\hat{B}_1+\Pi_2\hat{B}_2\big]^\top Y(t)-R_2^{-1}(t)\hat{B}_2^\top\tilde{\varphi}(t),
\end{equation}
where $\tilde{\varphi}(\cdot)$ satisfies \eqref{tildevarphi}, and the optimal state trajectory $(Y(\cdot),Z(\cdot))$ satisfies the BSDE
\begin{equation}\label{BSDE3}
\left\{
\begin{aligned}
-dY(t)&=\Big\{\big[\hat{A}_1+\hat{F}_2\Pi_2-\hat{B}_2R_2^{-1}(t)\big(\hat{B}_1+\Pi_2\hat{B}_2\big)^\top\big]Y(t)+\hat{C}_1^\top Z(t)\\
      &\qquad+(\hat{F}_2-\hat{B}_2R_2^{-1}(t)\hat{B}_2^\top)\tilde{\varphi}(t)\Big\}dt-Z(t)dW(t),\\
  Y(T)&=\hat{\xi}.
\end{aligned}
\right.
\end{equation}
\end{theorem}
\begin{proof}
For given $\xi$, let $\Pi_1(\cdot)$ satisfy \eqref{Pi1}, by the standard BSDE theory we can solve \eqref{phi} to obtain $(\tilde{\phi}(\cdot),\tilde{\eta}(\cdot))$. Let $\Pi_2(\cdot)$ satisfy \eqref{Pi2}, and by the standard SDE theory we can solve \eqref{tildevarphi} to obtain $\tilde{\varphi}(\cdot)$. Thus the feedback representation \eqref{LC2} can be obtained from the BSDE \eqref{BSDE3}. The proof is complete.
\end{proof}

We can also draw a flow chart to show the process in which we could represent the optimal feedback $\bar{u}_2(\cdot)$.\\

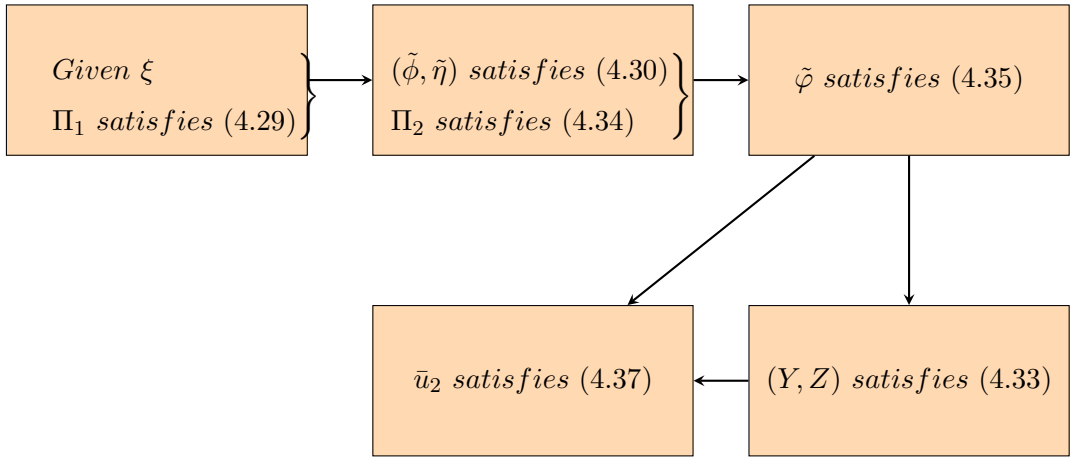
\begin{tikzpicture}[node distance=1cm]
\node (start) [startstop]
{
$$\left.
\begin{aligned}
&Given\ \xi\\
&\Pi_1\ satisfies\ \eqref{Pi1}
\end{aligned}
\right\}$$
};
\node (process1) [process, right of =start, xshift=4cm]
{
$$\left.
\begin{aligned}
&(\tilde{\phi},\tilde{\eta})\ satisfies\ \eqref{phi}\\
&\Pi_2\ satisfies\ \eqref{Pi2}
\end{aligned}
\right\}$$
};
\node (process2) [process,right of =process1,xshift=4cm] {$\tilde{\varphi}\ satisfies\ \eqref{tildevarphi}$};
\node (process3) [process,below of=process2,yshift=-3cm] {$(Y,Z)\ satisfies\ \eqref{X1Y1}$};
\node (process4) [process,below of =process3,right of=process2,xshift=-6cm,yshift=-3cm] {$\bar{u}_2\ satisfies\ \eqref{LC2}$};

\draw [arrow] (start) --(process1);
\draw [arrow] (process1) --(process2);
\draw [arrow] (process2) --(process3);
\draw [arrow] (process3) --(process4);
\draw [arrow] (process2) -- (process4);
\end{tikzpicture}

In the rest part of this section, we concentrate on the solvability of the two Riccati equations \eqref{Pi1}, \eqref{Pi2} of $\Pi_1(\cdot)$ and $\Pi_2(\cdot)$. For simplicity, we will just consider the constant coefficient case.

We first discuss the solvability for \eqref{Pi1} of $\Pi_1(\cdot)$. However, in this paper we consider only the case of $C=0$, and it is easy to get that $\hat{C}_1=\hat{D}_1=0$. Then we can rewrite \eqref{Pi1} to another form:
\begin{equation}\label{pi1}
\left\{
\begin{aligned}
&\dot{\Pi}_1(t)+(\hat{A}_1-\hat{B}_2R_2^{-1}\hat{B}_1^\top)\Pi_1(t)+\Pi_1(t)(\hat{A}_1^\top-\hat{B}_1R_2^{-1}\hat{B}_2^\top)\\
&\ +\Pi_1(t)(\hat{B}_1R_2^{-1}\hat{B}_1^\top-\hat{F}_1)\Pi_1(t)+\hat{B}_2R_2^{-1}\hat{B}_2^\top-\hat{F}_{2}=0,\\
&\Pi_1(T)=0.
\end{aligned}
\right.
\end{equation}

Then according to Theorem 5.2 of Yong \cite{Yong99}, we can obtain the following proposition.
\begin{mypro}
Let {\bf (L1)}, {\bf (L2)} hold, $C=0$, and ${\rm det}\bigg\{\begin{pmatrix}
0,&I
\end{pmatrix}e^{\mathcal{A}t}\begin{pmatrix}
0\\
I
\end{pmatrix}
\bigg\}>0,\ t\in[0,T]$ hold.
Then, \eqref{pi1} admits a unique solution $\Pi_1(\cdot)$ which has the following representation
\begin{equation}
\Pi_1(t)=-\Bigg [\begin{pmatrix}0,&I
\end{pmatrix}e^{\mathcal{A}(T-t)}\begin{pmatrix}
0\\
I
\end{pmatrix}\Bigg]^{-1}\begin{pmatrix}0,&I
\end{pmatrix}e^{\mathcal{A}(T-t)}\begin{pmatrix}
I\\
0
\end{pmatrix},
\end{equation}
where
\begin{equation}
\mathcal{A}\triangleq
\begin{pmatrix}
\hat{A}_1^\top-\hat{B}_1R_2^{-1}\hat{B}_2^\top&\hat{B}_1R_2^{-1}\hat{B}_1^\top-\hat{F}_1\\
\hat{F}_2-\hat{B}_2R_2^{-1}\hat{B}_2^\top&-\hat{A}_1+\hat{B}_2R_2^{-1}\hat{B}_1^\top
\end{pmatrix}.
\end{equation}
\end{mypro}
Then, we use the similar method to discuss the solvability for the Riccati equation \eqref{Pi2} of $\Pi_2(\cdot)$. Firstly we make the time reversing transformation:
\begin{equation*}
\tau=T-t,\ t\in[0,T].
\end{equation*}
Then, in the case of $C=0$, we can also obtain an equivalent form of \eqref{Pi2}:
\begin{equation}\label{tPi2}
\left\{
\begin{aligned}
&\dot{\Pi}_2(t)+\Pi_2(t)(\hat{A}_1-\hat{B}_2R_2^{-1}\hat{B}_1^\top)+(\hat{A}_1^\top-\hat{B}_1R_2^{-1}\hat{B}_2^\top)\Pi_2(t)\\
&\ +\Pi_2(t)(\hat{F}_2-\hat{B}_2R_2^{-1}\hat{B}_2^\top)\Pi_2(t)+\hat{F}_1-\hat{B}_1R_2^{-1}\hat{B}_1^\top=0,\\
&\Pi_2(T)=\hat{G}_2,
\end{aligned}
\right.
\end{equation}

Next, we introduce the following Riccati equation:
\begin{equation}\label{Pi21}
\left\{
\begin{aligned}
&\dot{\Pi}_{2,1}(t)+\Pi_{2,1}(t)\big[\hat{A}_1-\hat{B}_2R_2^{-1}\hat{B}_1^\top+(\hat{F}_2-\hat{B}_2R_2^{-1}\hat{B}_2^\top)\hat{G}_2\big]\\
&\ +\big[\hat{A}_1^\top-\hat{B}_1R_2^{-1}\hat{B}_2^\top+\hat{G}_2(\hat{F}_2-\hat{B}_2R_2^{-1}\hat{B}_2^\top)\big]\Pi_{2,1}(t)\\
&\ +\Pi_{2,1}(t)(\hat{F}_2-\hat{B}_2R_2^{-1}\hat{B}_2^\top)\Pi_{2,1}(t)+\hat{G}_2(\hat{A}_1-\hat{B}_2R_2^{-1}\hat{B}_1^\top)+(\hat{A}_1^\top-\hat{B}_1R_2^{-1}\hat{B}_2^\top)\hat{G}_2\\
&\ +\hat{G}_2(\hat{F}_2-\hat{B}_2R_2^{-1}\hat{B}_2^\top)\hat{G}_2+\hat{F}_{1}-\hat{B}_1R_2^{-1}\hat{B}_1^\top=0,\\
&\Pi_{2,1}(T)=\hat{G}_2.
\end{aligned}
\right.
\end{equation}
It is easy to see the solution $\Pi_{2,1}(\cdot)$ to \eqref{Pi21} and that $\Pi_2(\cdot)$ to \eqref{tPi2} are related by the following:
\begin{equation}\label{relation}
\Pi_2(t)=\hat{G}_2+\Pi_{2,1}(t),\ \ t\in[0,T].
\end{equation}
Therefore, similarly to the above discussion which is proved in \cite{Yong99}, we have the following result.
\begin{mypro}
Let {\bf (L1)}, {\bf (L2)}, {\bf (L3)} hold, $C=0$, and ${\rm det}\bigg\{\begin{pmatrix}
0,&I
\end{pmatrix}e^{\mathcal{B}t}\begin{pmatrix}
0\\
I
\end{pmatrix}
\bigg\}>0,\ t\in[0,T]$ hold. Then, \eqref{Pi21} admit unique solution $\Pi_{2,1}(\cdot)$ which has the following representation
\begin{equation}
\Pi_{2,1}(t)=-\Bigg [\begin{pmatrix}0,&I
\end{pmatrix}e^{\mathcal{B}(T-t)}\begin{pmatrix}
0\\
I
\end{pmatrix}\Bigg]^{-1}\begin{pmatrix}0,&I
\end{pmatrix}e^{\mathcal{B}(T-t)}\begin{pmatrix}
I\\
0
\end{pmatrix},
\end{equation}
where we let
\begin{equation}
\left\{
\begin{aligned}
\Phi&\triangleq\hat{A}_1-\hat{B}_2R_2^{-1}\hat{B}_1^\top+(\hat{F}_2-\hat{B}_2R_2^{-1}\hat{B}_2^\top)\hat{G}_2,\\
\Psi&\triangleq\hat{G}_2(\hat{A}_1-\hat{B}_2R_2^{-1}\hat{B}_1^\top)+(\hat{A}_1^\top-\hat{B}_1R_2^{-1}\hat{B}_2^\top)\hat{G}_2\\
    &\quad+\hat{G}_2(\hat{F}_2-\hat{B}_2R_2^{-1}\hat{B}_2^\top)\hat{G}_2+\hat{F}_{1}-\hat{B}_1R_2^{-1}\hat{B}_1^\top,\\
\mathcal{B}&\triangleq
\begin{pmatrix}
\Phi&\hat{F}_2-\hat{B}_2R_2^{-1}\hat{B}_2^\top\\
\Psi&-\Phi^\top
\end{pmatrix}.
\end{aligned}
\right.
\end{equation}
Furthermore, \eqref{relation} gives the solution $\Pi_2(\cdot)$ to the Riccati equation \eqref{Pi2}.
\end{mypro}

Finally, the optimal strategy $\bar{u}_1(\cdot)$ of the follower can also be represented in a similar way as $\bar{u}_2(\cdot)$. In fact, by \eqref{FC1}, \eqref{XY}, we have
\begin{equation}\label{FC2}
\begin{aligned}
\bar{u}_1(t)&=-R_1^{-1}(t)B_1(t)^\top[P_2(t)\bar{y}(t)+\bar{\varphi}(t)]\\
            &=-R_1^{-1}(t)B_1(t)^\top\big[(0,P_2(t))Y(t)+(I,0)X(t)\big]\\
            &=-R_1^{-1}(t)B_1(t)^\top\big[(0,P_2)+(I,0)\Pi_2\big]Y(t)-R_1^{-1}(t)B_1(t)^\top(I,0)\tilde{\varphi}(t).
\end{aligned}
\end{equation}

Up to now, we have obtained the optimal solution $(\bar{u}_1(\cdot),\bar{u}_2(\cdot))$ to the LQ Stackelberg game of BSDEs, whose state feedback representation are given by \eqref{FC2} and \eqref{LC2}.

\section{Application in finance}

In this section, we consider an optimal consumption rate problem of two players in the financial market, which can be formulated as an LQ Stackelberg game of BSDEs. Then the theoretic results in the previous sections can be applied, and in fact, it motivates us the study of our problems.

Suppose in the financial market, the investors have two tradable assets. One is a risk-free asset (bond or bank account) whose price dynamic is given by the following ordinary differential equation (ODE):
\begin{equation}
dS_0(t)=r(t)S_0(t)dt,\ t\geq0,\ S_0(0)=s_0,
\end{equation}
where $r(t)$ is called the interest rate. The other one is a risky asset (stock or investment fund) whose price dynamic is subject to the following SDE:

\begin{equation}
dS_1(t)=S_1(t)[\mu(t)dt+\sigma(t)dW(t)],\ t\geq0,\ S_1(0)=s_1,
\end{equation}
where $\mu(\cdot)$ and $\sigma(\cdot)$ are called the instantaneous rate of return and instantaneous volatility, respectively. In this section, we assume that the above market coefficients $r(\cdot)$, $\mu(\cdot)$ and $\sigma(\cdot)$ are deterministic and bounded functions, and $\mu(t)\geq r(t)$ for any $t\geq0$.

Suppose there are two agents (players) working together to invest the bond and the stock, whose decision has no influence on the prices in the financial market. One of the agents is the follower, who has an instantaneous consumption rate $c_1(\cdot)$, and the other one is the leader, who has an instantaneous consumption rate $c_2(\cdot)$.

Now, the two players wish to achieve a terminal wealth goal $\xi$ at the terminal time $T$, where $\xi$ is an $\mathcal{F}_T$-measurable non-negative square-integrable random variable. We use $\pi(\cdot)$ to represent the amount that the two players invest in the stock. Then the value $y(\cdot)$ of their wealth process is modeled by
\begin{equation}
\left\{
\begin{aligned}
&dy(t)=\big[r(t)y(t)+(\mu(t)-r(t))\pi(t)-c_1(t)-c_2(t)\big]dt+\pi(t)\sigma(t)dW(t),\ t\in[0,T],\\
&y(T)=\xi.
\end{aligned}
\right.
\end{equation}
If we set $z(\cdot)=\pi(\cdot)\sigma(\cdot)$, then we can rewrite the above equation as
\begin{equation}\label{y1}
\left\{
\begin{aligned}
-dy(t)&=\big[-r(t)y(t)-\frac{\mu(t)-r(t)}{\sigma(t)}z(t)+c_1(t)+c_2(t)\big]dt-z(t)dW(t),\ t\in[0,T],\\
  y(T)&=\xi.
\end{aligned}
\right.
\end{equation}
Let $\mathcal{U}_i=\big\{c_i(\cdot)\big|c_i(\cdot)\in L_\mathcal{F}^2(0,T;\mathbb{R}),t\in[0,T]\big\},i=1,2$, and each $c_i(\cdot)\in\mathcal{U}_i(\cdot)$ is called an admissible control. For any $(c_1(\cdot),c_2(\cdot))\in\mathcal{U}_1\times\mathcal{U}_2$, the BSDE \eqref{y1} admit a unique solution pair $(y(\cdot),z(\cdot))$ in $L_\mathcal{F}^2(0,T;\mathbb{R})\times L_\mathcal{F}^2(0,T;\mathbb{R})$.

Now, we define the associated expected utility functionals
\begin{equation}\label{cf example}
J_i(u_1(\cdot),u_2(\cdot);\xi)=\frac{1}{2}\mathbb{E}\int_0^TR_i(t)c_i^2(t)dt-G_iy^2(0),\ i=1,2,
\end{equation}
where $R_i(\cdot)>0$ is a deterministic function and $G_i\geq0$ is a constant. For the two players, it is natural to desire to maximize his expected utility functional representing cumulative consumption and to minimize the initial reserve.

This is an LQ Stackelberg game of BSDEs in the financial market. The target of this section is to find the optimal solution $(\bar{c}_1(\cdot),\bar{c}_2(\cdot))\in\mathcal{U}_1\times\mathcal{U}_2$, which is the Stackelberg equilibrium point, as in Definition 2.1.

Obviously, this problem can be regarded as a special case of that in Section 4. Therefore, we can use the results (Theorems 4.2 and 4.4) to solve it. Noting \eqref{BSDE}, \eqref{cf1} and \eqref{cf2}, we know in this section $A(t)=-r(t)$, $B_1(t)=B_2(t)=1$, $C(t)=-\frac{\mu(t)-r(t)}{\sigma(t)}$, $Q_1(t)=Q_2(t)=0$ and $S_1(t)=S_2(t)=0$, for any $t\in[0,T]$.

Firstly, we solve the follower's problem. For given $\xi$ and any $c_2(\cdot)\in\mathcal{U}_2$, using Theorem 4.2, we can get
\begin{equation}\label{c1}
\bar{c}_1(t)=-\frac{P_2(t)\bar{y}(t)+\varphi(t)}{R_1(t)},
\end{equation}
where $(\bar{y}(\cdot),\bar{z}(\cdot))$ satisfy the following BSDE:
\begin{equation}\label{y1optimal}
\left\{
\begin{aligned}
-d\bar{y}(t)&=\bigg\{-\Big[\frac{P_2(t)}{R_1(t)}+r(t)\Big]\bar{y}(t)-\frac{\mu(t)-r(t)}{\sigma(t)}z(t)-\frac{\varphi(t)}{R_1(t)}+c_2(t)\bigg\}dt-\bar{z}(t)dW(t),\\
  \bar{y}(T)&=\xi,
\end{aligned}
\right.
\end{equation}
$P_1(\cdot)$ and $P_2(\cdot)$ satisfy the following Riccati equations:
\begin{equation}\label{R1 example}
\left\{
\begin{aligned}
&\dot{P}_1(t)+\bigg[\Big(\frac{\mu(t)-r(t)}{\sigma(t)}\Big)^2-2r(t)\bigg]P_1(t)+\frac{1}{R_1(t)}=0,\\
&P_1(T)=0,
\end{aligned}
\right.
\end{equation}
\begin{equation}\label{R2 example}
\left\{
\begin{aligned}
&\dot{P}_2(t)+\bigg[\frac{1}{R_1(t)}+P_1(t)\Big(\frac{\mu(t)-r(t)}{\sigma(t)}\Big)^2\bigg]P_2^2(t)+2r(t)P_2(t)=0,\\
&P_2(0)=G_1,
\end{aligned}
\right.
\end{equation}
respectively, and $(\varphi(\cdot),\phi(\cdot),\eta(\cdot))$ satisfy the following FBSDE:
\begin{equation}\label{FBSDE example}
\left\{
\begin{aligned}
d\varphi(t)&=\bigg\{\Big[-\frac{P_2(t)}{R_1(t)}+\Big(\frac{\mu(t)-r(t)}{\sigma(t)}\Big)^2P_1(t)P_2(t)-r(t)\Big]\varphi(t)\\
           &\qquad+\frac{\mu(t)-r(t)}{\sigma(t)}P_2(t)\eta(t)+P_2(t)c_2(t)\bigg\}dt\\
           &\quad+\bigg\{-\frac{\mu(t)-r(t)}{\sigma(t)}\varphi(t)+\frac{\mu(t)-r(t)}{\sigma(t)}P_2(t)\phi(t)+P_2(t)\eta(t)\bigg\}dW(t),\\
  -d\phi(t)&=\bigg\{-r(t)\phi(t)-\frac{\mu(t)-r(t)}{\sigma(t)}\eta(t)-c_2(t)\bigg\}dt-\eta(t)dW(t),\\
\varphi(0)&=0,\ \phi(T)=-\xi.
\end{aligned}
\right.
\end{equation}

Next, we solve the leader's problem. Noting \eqref{definition} and putting
\begin{equation*}
\hat{A}_1(t)=
\begin{pmatrix}
-r(t)-\frac{P_2(t)}{R_1(t)}&0\\
0&-r(t)-\frac{P_2(t)}{R_1(t)}
\end{pmatrix},\ \
\hat{B}_1(t)=
\begin{pmatrix}
P_2(t)\\
0
\end{pmatrix},\ \
\hat{B}_2=
\begin{pmatrix}
0\\
1
\end{pmatrix},\ \
\end{equation*}
\begin{equation*}
\hat{C}_1(t)=
\begin{pmatrix}
\frac{-1-P_2^2(t)P_1^2(t)-P_2(t)P_1(t)}{P_1(t)P_2(t)+1}\frac{\mu(t)-r(t)}{\sigma(t)}&0\\
0&-\frac{\mu(t)-r(t)}{\sigma(t)}
\end{pmatrix},
\end{equation*}
\begin{equation*}
\hat{D}_1(t)=
\begin{pmatrix}
0&-P_2(t)\frac{\mu(t)-r(t)}{\sigma(t)}\\
-P_2(t)\frac{\mu(t)-r(t)}{\sigma(t)}&0
\end{pmatrix},
\end{equation*}
\begin{equation*}
\end{equation*}\begin{equation*}
\hat{F}_1(t)=
\begin{pmatrix}
0&\big(\frac{\mu(t)-r(t)}{\sigma(t)}\big)^2P_2^2(t)P_1(t)\\
\big(\frac{\mu(t)-r(t)}{\sigma(t)}\big)^2P_2^2(t)P_1(t)&0
\end{pmatrix},\ \
\end{equation*}
\begin{equation*}
\hat{F}_2(t)=
\begin{pmatrix}
0&-\frac{1}{R_1(t)}\\
-\frac{1}{R_1(t)}&0
\end{pmatrix},\ \
\hat{S}_1(t)=
\begin{pmatrix}
0&-P_2(t)\\
-P_2(t)&0
\end{pmatrix},
\end{equation*}
by Theorem 4.4, we can get
\begin{equation}\label{c2}
\bar{c}_2(t)=-\frac{\big[\hat{B}_1(t)+\Pi_2(t)\hat{B}_2(t)\big]^\top Y(t)}{R_2(t)}-\frac{\hat{B}_2^\top\tilde{\varphi}(t)}{R_2(t)},
\end{equation}
where $(\bar{Y}(\cdot),\bar{Z}(\cdot))$ satisfy the following 2-dimensional BSDE:
\begin{equation}\label{YZ example}
\left\{
\begin{aligned}
-dY(t)&=\Bigg\{\bigg[\begin{pmatrix}-r(t)-\frac{P_2(t)}{R_1(t)}&0\\-\frac{P_2(t)}{R_2(t)}&-r(t)-\frac{P_2(t)}{R_1(t)}\end{pmatrix}
       +\begin{pmatrix}0&-\frac{1}{R_1(t)}\\-\frac{1}{R_1(t)}&-\frac{1}{R_2(t)}\end{pmatrix}\Pi_2(t)\bigg]Y(t)\\
      &\qquad+\begin{pmatrix}\frac{-1-P_2^2(t)P_1^2(t)-P_2(t)P_1(t)}{P_1(t)P_2(t)+1}\frac{\mu(t)-r(t)}{\sigma(t)}&0\\0&-\frac{\mu(t)-r(t)}{\sigma(t)}\end{pmatrix}Z(t)\\
      &\qquad+\begin{pmatrix}0&-\frac{1}{R_1(t)}\\-\frac{1}{R_1(t)}&-\frac{1}{R_2(t)}\end{pmatrix}\tilde{\varphi}(t)\Bigg\}dt-Z(t)dW(t),\\
  Y(T)&=\hat{\xi},
\end{aligned}
\right.
\end{equation}
$\Pi_1(\cdot)$ and $\Pi_2(\cdot)$ satisfy the 2-dimensional Riccati equations \eqref{Pi1}, \eqref{Pi2} respectively, and $(\tilde{\varphi}(\cdot),\tilde{\phi}(\cdot),\tilde{\eta}(\cdot))$ satisfy the 2-dimensional FBSDE \eqref{tildevarphi}, \eqref{phi}.

Thus, $(\bar{c}_1(\cdot),\bar{c}_2(\cdot))$ determined by \eqref{c1} and \eqref{c2} is a Stackelberg equilibrium point of our game problem of BSDEs. Moreover, by a dual technique similar to \cite{EPQ97}, from \eqref{YZ} we have
\begin{equation}\label{explicit}
Y(t)=\mathbb{E}\Bigg[\hat{\xi}\Gamma_t(T)+\int_t^T\Bigg\{\begin{pmatrix}0&-\frac{1}{R_1(s)}\\-\frac{1}{R_1(s)}&-\frac{1}{R_2(s)}\end{pmatrix}\tilde{\varphi}(s)\Bigg\}\Gamma_t(s)ds\Bigg|\mathcal{F}_t\Bigg],
\end{equation}
where for $t\in[0,T]$, $\Gamma_t(\cdot)$ is the unique solution to
\begin{equation*}
\left\{
\begin{aligned}
d\Gamma_t(s)&=\bigg[\begin{pmatrix}-r(s)-\frac{P_2(s)}{R_1(s)}&0\\-\frac{P_2(s)}{R_2(s)}&-r(s)-\frac{P_2(s)}{R_1(s)}\end{pmatrix}
             +\begin{pmatrix}0&-\frac{1}{R_1(s)}\\-\frac{1}{R_1(s)}&-\frac{1}{R_2(s)}\end{pmatrix}\Pi_2(s)\bigg]\Gamma_t(s)ds\\
            &\quad+\begin{pmatrix}\frac{-1-P_2^2(s)P_1^2(s)-P_2(s)P_1(s)}{P_1(s)P_2(s)+1}\frac{\mu(s)-r(s)}{\sigma(s)}&0\\0&-\frac{\mu(s)-r(s)}{\sigma(s)}\end{pmatrix}\Gamma_t(s)dW(s),\ s\in[t,T],\\
 \Gamma_t(t)&=1,
\end{aligned}
\right.
\end{equation*}
or explicitly,
\begin{equation}\label{Gamma}
\begin{aligned}
\Gamma_t(s)&=\exp\Bigg\{\int_t^s\bigg[\begin{pmatrix}-r(\tau)-\frac{P_2(\tau)}{R_1(\tau)}&0\\-\frac{P_2(\tau)}{R_2(\tau)}&-r(\tau)-\frac{P_2(\tau)}{R_1(\tau)}\end{pmatrix}
             +\begin{pmatrix}0&-\frac{1}{R_1(\tau)}\\-\frac{1}{R_1(\tau)}&-\frac{1}{R_2(\tau)}\end{pmatrix}\Pi_2(\tau)\\
           &\qquad\qquad\qquad+\frac{1}{2}\begin{pmatrix}\frac{-1-P_2^2(\tau)P_1^2(\tau)-P_2(\tau)P_1(\tau)}{P_1(\tau)P_2(\tau)+1}
            \frac{\mu(\tau)-r(\tau)}{\sigma(\tau)}&0\\0&-\frac{\mu(\tau)-r(\tau)}{\sigma(\tau)}\end{pmatrix}^2\bigg]d\tau\\
           &\qquad\quad+\int_t^s\begin{pmatrix}\frac{-1-P_2^2(\tau)P_1^2(\tau)-P_2(\tau)P_1(\tau)}{P_1(\tau)P_2(\tau)+1}
            \frac{\mu(\tau)-r(\tau)}{\sigma(\tau)}&0\\0&-\frac{\mu(\tau)-r(\tau)}{\sigma(\tau)}\end{pmatrix}\Gamma_t(\tau)dW(\tau)\Bigg\},\ s\in[t,T].\\
\end{aligned}
\end{equation}
Then the optimal initial wealth reserve $\bar{y}(0)$ is the second component of the following 2-dimensional vector
\begin{equation}\label{initial reserve}
Y(0)=\mathbb{E}\Bigg[\hat{\xi}\Gamma_0(T)+\int_0^T\Bigg\{\begin{pmatrix}0&-\frac{1}{R_1(t)}\\-\frac{1}{R_1(t)}&-\frac{1}{R_2(t)}\end{pmatrix}\tilde{\varphi}(t)\Bigg\}\Gamma_0(t)dt\Bigg].
\end{equation}

\section{Concluding Remarks}

In this paper, we have studied a new class of Stackelberg game of BSDEs, with deterministic coefficients and convex control domain. Necessary and sufficient conditions of the optimality for the follower and the leader are first given for the general problem. Then an LQ Stackelberg game of BSDEs is investigated under standard assumptions. The state feedback representation for the optimal control of the follower is first given via two Riccati equations, a BSDE and a SDE. Then the leader's problem is formulated as an optimal control problem of FBSDE with the control-independent diffusion term. Two high-dimensional Riccati equations, a high-dimensional BSDE and a high-dimensional SDE are introduced to represent the state feedback for the optimal control of the leader. The solvability of the four Riccati equations are discussed. Theoretic results are applied to an optimal consumption rate problem of two players in the financial market.

An outstanding open problem to study is the Stackelberg game of BSDEs where all the coefficients are random (as in \cite{Yong02}). In this case, the Riccati equations become nonlinear BSDEs (rather than ODEs as in this paper), and the solvability of them is very challenging to prove.

\end{document}